\documentclass{article}
\usepackage{amssymb}
\usepackage{amsmath}
\usepackage{amsfonts}
\usepackage{amsthm}

\newtheorem{theorem}{Theorem}[section]
\newtheorem{lemma}[theorem]{Lemma}
\newtheorem{corollary}[theorem]{Corrolary}
\newtheorem{remark}[theorem]{Remark}
\input{tcilatex}

\begin{document}

\title{A new closed formula for the Hermite interpolating polynomial with
applications on the spectral decomposition of a matrix}
\author{Aristides I. Kechriniotis \thanks{%
Department of Informatics, Technological Educational Institute of Lamia,
35100 Lamia, Greece} \and Konstantinos K. Delibasis \thanks{%
University of Central Greece, Dept. of Biomedical Informatics, Greece} \and %
Christos Tsonos \thanks{%
Department of Electronics, Technological Educational Institute of Lamia,
35100 Lamia, Greece} \and Nicholas Petropoulos \thanks{%
Department of Electronics, Technological Educational Institute of Lamia,
35100 Lamia, Greece}}
\date{December 19, 2011}
\maketitle

\begin{abstract}
We present a new closed form for the interpolating polynomial of the general
univariate Hermite interpolation that requires only calculation of
polynomial derivatives, instead of derivatives of rational functions. This
result is used to obtain a new simultaneous polynomial division by a common
divisor over a perfect field. The above findings are utilized to obtain a
closed formula for the semi--simple part of the Jordan decomposition of a
matrix. Finally, a number of new identities involving polynomial derivatives
are obtained, based on the proposed simultaneous polynomial division. The
proposed explicit formula for the semi--simple part has been implemented
using the Matlab programming environment.
\end{abstract}

%\keyphrases{semisimple part of a matrix, Hermite interpolation, Euclidean polynomial
%division }
%\AMclass{\qquad 12Y05; 15A21; 11C08; 65F30}

\section{Introduction}

\numberwithin{equation}{section}

The Hermite interpolation of total degree is described in the following
Theorem \cite{Berezin-Zhidkov}:

\begin{theorem}
\label{thm:1.1} Given $n$ distinct elements $\lambda _{0},\lambda
_{1},\ldots,\lambda _{n-1}$ in a perfect field $\Bbbk$, positive integers~$%
m_{i}$, $i=0,\ldots,n-1$, and $a_{ij}\in \Bbbk$ for $0\leq i\leq n-1,\ 0\leq
j\leq m_{i}-1,$ then there exists one and only one polynomial $r\in \Bbbk [x]
$ of degree less than $\sum_{i=0}^{n-1}m_{i}$, such that 
\begin{equation}
r^{(j)}(\lambda _{i}) =a_{ij},~0\leq j\leq m_{i}-1,~0\leq i\leq n-1.
\end{equation}
This polynomial $r$ is explicitly given by, 
\begin{eqnarray*}
r(x)&=&\sum_{i=0}^{n-1}\sum_{j=0}^{m_{i}-1}\sum_{k=0}^{m_{i}-j-1}a_{ij}\frac{%
1}{ j!}\frac{1}{k!}\left[\frac{(x-\lambda _{i}) ^{m_{i}}}{\Omega(x) }\right]%
_{x=x_{i}}^{(k) } \\
&&\times \frac{\Omega (x) }{(x-\lambda _{i})^{m_{i}-j-k}},
\end{eqnarray*}
where 
\begin{equation*}
\Omega(x) =\prod_{i=0}^{n-1}(x-\lambda _{i})^{m_{i}},
\end{equation*}
and $r^{(k)}(a) $ is the $k-$th derivative of $r$ at $a$.
\end{theorem}

The applications of hermite interpolation to numerical analysis is well
known. In this paper we show the use of a computationally more efficient
formula for the Hermite interpolation to a number of algebraic applications.

A number of forms of the interpolating polynomial $r(x)$ have been reported
in the literature, however they all require calculation of derivatives of
rational polynomial functions (e.g. \cite{Berezin-Zhidkov}), or recursive
calculation of the coefficients $a_{ij}$ of Theorem~\ref{thm:1.1} (e.g. \cite%
{Sakai-Vertesi}). We firstly propose a new closed form of the interpolating
polynomial $r$ of the univariate Hermite interpolation that requires only
calculation of polynomial derivatives instead of derivatives of rational
functions, as required in Theorem~\ref{thm:1.1}, thus achieving a
substantial computational acceleration compared to the Theorem~\ref{thm:1.1}%
, without imposing any restrictions in the generality of the interpolation.
We then exploit this result to show a Theorem for the calculation of the
remainders of the simultaneous division of a number of polynomials by a
separable polynomial divisor, as derivatives of a unique polynomial. Namely,
let us note that the remainder $r \in \Bbbk[x] $ of the Euclidean division
of any polynomial $P \in \Bbbk[x] $ of degree $n$ by a separable polynomial $%
Q\in\Bbbk[x] $ of degree $m$, where $n\geq m$, can be calculated in closed
form using the Langrange interpolation formula as following: 
\begin{equation*}
r(x) = \sum_{i=1}^{m} P(\lambda _{i}) \prod_{\substack{ j=1  \\ j\neq i}}^{m}%
\frac{(x-\lambda _{j})} {(\lambda_{i}-\lambda _{j})},
\end{equation*}
where $\lambda _{1},\ldots,\lambda _{m}$ are the roots of $Q$ in the
algebraic closure $\overline{\Bbbk}$ of $\Bbbk$. In this work we extent this
simplified idea of polynomial remainder calculation by Langrange
interpolation, to achieve simultaneous polynomial division by a common
separable divisor, using the closed form of the interpolating polynomial of
the univariate Hermite interpolation $r$ that we propose.

Furthermore, the above results will be used to obtain a new closed formula
for the semi--simple part of the Jordan decomposition of an algebraic
element in an arbitrary algebra. If $\Bbbk $ is a perfect field and $\mathbf{%
A}$ an algebra over $\Bbbk $ with unit $1$, then the well--known result on
the Jordan decomposition, (e.g. \cite{Bourbaki, Hoffman-Kunze, Mulders}) is
presented in the following theorem.

\begin{theorem}
\label{thm:1.2} For each algebraic $A\in \mathbf{A}$ there exist unique $%
S_{A},N_{A}\in \Bbbk[\mathbf{A}] $ such that $A=S_{A}+N_{A}$, $S_{A}$ is
semi--simple and $N_{A}$ is nilpotent.
\end{theorem}

It should be noted that in the case of $\Bbbk$ being algebraically closed
and $\mathbf{A}$ is a subalgebra of the matrix-algebra $\Bbbk ^{n\times n}$,
then $S_{A}$ is semi--simple if and only if $S_{A}$ is diagonalizable.

Theorem~\ref{thm:1.2} has various interesting applications, such as
efficient computation of high powers of $A$, therefore, efficient algorithms
for the calculation of $S_{A}$ in terms of $A$ are of great importance. A
proof of the existence of $S_{A}$ that is presented in the book of Hoffman
and Kunze \cite{Hoffman-Kunze}, is based on the Newton's method and yields
direct methods for computations. An algorithm, which is essentially based on
these ideas, is given by Levelt \cite{Levelt}. The algorithm of Bourgoyne
and Cushman \cite{Burgoyne-Cushman} is faster, because higher derivatives
are used. In \cite{Schmidt} D. Schmidt has used the Newton's method to
construct the semi--simple part of the Jordan decomposition of an algebraic
element in an arbitrary algebra, showing quadratic convergence of the
algorithm. Another approach uses the partial fractions decomposition of the
reciprocal of the minimal polynomial \cite{Hsieh-Kohno-Sibuya, Sobczyk,
Verde-Star}. These works however, do not provide a closed form formula for $%
S_{A}$ as a polynomial of $A$. An explicit construction of the spectral
decomposition of a matrix using Hermite interpolation appears in \cite%
{Verde-Star}. However, this approach requires the use of Taylor coefficients
of the reciprocal of the matrix minimal polynomial.

In this work, we exploit our proposed simultaneous polynomial division by a
common separable divisor, to obtain a new closed formula for the
semi--simple part $S_{A}$ of the Jordan decomposition of an algebraic
element $A$ in an arbitrary algebra. The proposed closed formula requires
only evaluation of the derivatives of the basic Hermite--like interpolation
polynomials that are associated by the eigenvalues of $A$, up to the maximum
algebraic multiplicity of the roots of the minimal polynomial of $A$, as
well as matrix multiplication operations. The derived closed formula for the
semi--simple part of the Jordan decomposition has been implemented in the
Matlab programming environment. The sourse code is provided and results are
shown from it's executions on examples of matrices taken from the literature.

Furthermore, a number of new identities involving polynomial derivatives are
also shown, based on the the simultaneous polynomial division by a common
separable divisor.

The rest of the paper is organized as follows. In section 2 we present a new
closed form for Hermite interpolation. In section 3 we show a Theorem for
the calculation of the remainders of the simultaneous division of a number
of polynomials by a separable polynomial divisor, as derivatives of a unique
polynomial. In section 4 we present an explicit Formula for $S_{A}$, the
relative source code in MATLAB environment, as well as some numerical
examples from the literature. Finally, in section 5 the main result of
section 3 is generalized and applied to the derivation of some interesting
polynomial formulas.

\section{A new closed form for Hermite interpolation}

Let $\lambda _{0},\lambda _{1},\ldots,\lambda _{n-1\text{ }}$ be distinct
elements in a perfect field $\Bbbk $ and $m_{0},m_{1},\ldots,m_{n-1}$ be
positive integers. Let us denote by $L_{k}$ the polynomials given by 
\begin{equation}  \label{eq:2.1}
L_{k}(x) =\prod_{\substack{ i=0  \\ i\neq k}}^{n-1}\frac{( x-\lambda _{i})
^{m_{i}}}{( \lambda _{k}-\lambda _{i}) ^{m_{i}}}\in \Bbbk [x] .
\end{equation}
Furthernore, we denote by $\Lambda _{k},~0\leq k\leq n-1$ the $m_{k}\times
m_{k}$ lower triangular matrices $\left[ l_{ij}\right] \in \Bbbk
^{m_{k}\times m_{k}},~0\leq i,j\leq m_{k}-1$, given by 
\begin{equation*}
l_{ij}:=\left\{ 
\begin{array}{ccc}
\binom{i}{j}( L_{k}) ^{( i-j) }( \lambda _{k}) & \text{ \ \ \ \ \ if \ \ \ }
& 0\leq j\leq i\leq m_{k}-1 \\ 
0 & \text{ \ \ \ \ if \ \ } & 0\leq i<j\leq m_{k}-1%
\end{array}
\right. ,~
\end{equation*}
where $( L_{k}) ^{( i-j) }( \lambda _{k}) $ is the derivative of order $(
i-j) $ of the polynomial $L_{k}(x) $ at $\lambda _{k}$. Thus $\Lambda _{k}$
has the following representation 
\begin{equation}  \label{eq:2.2}
\left[ 
\begin{array}{cccccc}
\binom{0}{0}L_{k}( \lambda _{k}) & 0 & . & . & . & 0 \\ 
\binom{1}{0}( L_{k}) ^{( 1) }( \lambda _{k}) & \binom{1}{1}L_{k}( \lambda
_{k}) & . & . & . & 0 \\ 
. & . & . &  &  & . \\ 
. & . &  & . &  & . \\ 
. & . &  &  & . & . \\ 
\binom{m_{k}-1}{0}( L_{k}) ^{( m_{k}-1) }( \lambda _{k}) & \tbinom{m_{k}-1}{1%
}( L_{k}) ^{( m_{k}-2) }( \lambda _{k}) & . & . & . & \binom{m_{k}-1}{
m_{k}-1}L_{k}( \lambda _{k})%
\end{array}
\right] .
\end{equation}
\noindent For our purpose the following technical Lemmas are required:

\begin{lemma}
\label{lem:2.1} The matrices $\Lambda _{k},~0\leq k\leq n-1$ are invertible
with $\Lambda_{k}^{-1}=\sum_{i=0}^{m_{k}-1}( I_{m_{k}}-\Lambda _{k}) ^{i}$,
where $I_{m_{k}}$ is the $m_{k}\times m_{k}$ unit matrix.
\end{lemma}

\begin{proof}
Clearly, for $0\leq k\leq n-1$ holds $L_{k}( \lambda _{k}) =1.$
Therefore, all matrices $\Lambda _{k}$, $0\leq k\leq n-1$ are invertible and
unit diagonal lower triangular. Therefore, it follows $(
I_{m_{k}}-\Lambda _{k}) ^{m_{k}}=0$ and consequently $\Lambda
_{k}\sum_{i=0}^{m_{k}-1}( I_{m_{k}}-\Lambda _{k}) ^{i}=I_{m_{k}}.$
\end{proof}

Using the Leibnitz's rule for derivatives, we easily get the following Lemma:

\begin{lemma}
\label{lem:2.2} For $0\leq i,s\leq m_{k}-1$ and $0\leq j,t\leq n-1$, the
following holds: 
\begin{equation}  \label{eq:2.3}
( \frac{(x-\lambda _{t}) ^{s}}{s!}L_{t}(x) ) ^{( i) }\mid _{x=\lambda
_{j}}=\left\{ 
\begin{array}{ccc}
0 & \text{if} & t\neq j \\ 
0 & \text{if} & t=j\text{ and }i<s \\ 
\binom{i}{s}( L_{j}) ^{( i-s) }( \lambda _{j}) & \text{if} & t=j\text{ and }%
s\leq i%
\end{array}
\right. .
\end{equation}
\end{lemma}

Our proposed form of Hermite interpolation can now be presented in the
following Theorem.

\begin{theorem}
\label{thm:2.3} Given $n$ distinct elements $\lambda _{0},\lambda
_{1},\ldots,\lambda _{n-1}$ in a perfect field $\Bbbk $, positive integers~$%
m_{i}$, $i=0,\ldots,n-1$ , and $a_{ij}\in \Bbbk $ for $0\leq j\leq n-1,\
0\leq i\leq m_{j}-1,$ then there exists one and only one polynomial $r\in
\Bbbk [x] $ of degree less than $\sum_{i=0}^{n-1}m_{i}$, such that 
\begin{equation}  \label{eq:2.4}
r^{( i) }( \lambda _{j}) =a_{ij},~0\leq j\leq n-1,~0\leq i\leq m_{j}-1.
\end{equation}
This polynomial $r$ is explicitly given by, 
\begin{eqnarray}  \label{eq:2.5}
r &=&\sum_{j=0}^{n-1}X_{j}\Lambda _{j}^{-1}A_{j} \\
&=&\sum_{j=0}^{n-1}\sum_{k=0}^{m_{j}-1}X_{j}( I_{m_{j}}-\Lambda _{j})
^{k}A_{j},  \notag
\end{eqnarray}
where the matrices $X_{j}$,$\ A_{j}$ are given by 
\begin{equation*}
X_{j}=\left[ 
\begin{array}{cccccc}
L_{j}(x) & \frac{(x-\lambda _{j}) }{1!}L_{j}( x) & . & . & . & \frac{%
(x-\lambda _{j}) ^{m_{j}-1}}{( m_{j}-1) !}L_{j}(x)%
\end{array}
\right] \ ,
\end{equation*}
and 
\begin{equation*}
A_{j}=\left[ 
\begin{array}{cccccc}
a_{0j} & a_{ij} & . & . & . & a_{m_{j}-1j}%
\end{array}
\right]^{T}.
\end{equation*}
\end{theorem}

\begin{proof}
It can be observed that (\ref{eq:2.5}) can be equivalently written in
the following form: 
\begin{equation*}
r(x) =\sum_{j=0}^{n-1}\sum_{i=0}^{m_{j}-1}c_{ij}\frac{(
x-\lambda _{j}) ^{i}}{i!}L_{j}(x) \in \Bbbk[x], 
\end{equation*} 
where 
\begin{equation}\label{eq:2.6}
\left[ 
\begin{array}{c}
c_{0j} \\ 
c_{1j} \\ 
. \\ 
. \\ 
. \\ 
c_{m_{j}-1j} 
\end{array} 
\right] =\Lambda _{j}^{-1}\left[ 
\begin{array}{c}
a_{0j} \\ 
a_{1j} \\ 
. \\ 
. \\ 
. \\ 
a_{m_{j}-1j} 
\end{array} 
\right] ,0\leq j\leq n-1.\ 
\end{equation} 
Now, by calculating the derivative of order $i$ of the polynomial $r$ at 
$\lambda _{j}$ and using (\ref{eq:2.3}) in Lemma \ref{lem:2.2} we obtain 
\begin{equation}\label{eq:2.7}
r^{( i) }( \lambda _{j}) =\sum_{k=0}^{i}c_{kj}\binom{i}{k}( L_{j}) ^{( i-k) }( \lambda _{j}),
\end{equation} 
for all $0\leq j\leq n-1,$ $0\leq i\leq m-1$. Taking into consideration the
definition of $\Lambda _{j}$ in (\ref{eq:2.2}), the system of
equations (\ref{eq:2.7}) can be rewritten in the following matrix form 
\begin{equation}\label{eq:2.8}
\left[ 
\begin{array}{c}
r( \lambda _{j}) \\ 
r^{\prime }( \lambda _{j}) \\ 
. \\ 
. \\ 
. \\ 
r^{( m_{j}-1) }( \lambda _{j}) 
\end{array} 
\right] =\Lambda _{j}\left[ 
\begin{array}{c}
c_{0j} \\ 
c_{1j} \\ 
. \\ 
. \\ 
. \\ 
c_{m_{j}-1j} 
\end{array} 
\right] ,\ 0\leq j\leq n-1
\end{equation} 
By substitutting (\ref{eq:2.6}) in (\ref{eq:2.8}), we get 
\begin{equation*}
\left[ 
\begin{array}{c}
r( \lambda _{j}) \\ 
r^{\prime }(\lambda _{j}) \\ 
. \\ 
. \\ 
. \\ 
r^{( m_{j}-1) }(\lambda _{j}) 
\end{array} 
\right] =\left[ 
\begin{array}{c}
a_{0j} \\ 
a_{1j} \\ 
. \\ 
. \\ 
. \\ 
a_{m_{j}-1j} 
\end{array} 
\right] .
\end{equation*} 
Hence, we derive that $r^{( i)}(\lambda _{j})
=a_{ij},$ $0\leq j\leq n-1,0\leq i\leq m_{j}-1$. The second equality of (\ref{eq:2.5})
is obtained by Lemma \ref{lem:2.1}.

Moreover, it can be easily confirmed that any polynomial $(x-\lambda
_{j}) ^{i}L_{j}(x) ,$ $0\leq j\leq n-1$, $0\leq i\leq
m_{j}-1$ has degree less than $\sum_{i=0}^{n-1}m_{i}$, and since $r$ is a $ 
\Bbbk $- linear combination of these polynomials, we conclude that the
degree of $r$ is less than $\sum_{i=0}^{n-1}m_{i}$.
\end{proof}

\section{Simultaneous division of polynomials by separable polynomial}

At this point we are ready to present the following generalization of
Euclidean polynomial division by a separable divisor, based on Theorem $2.3$.

\begin{theorem}
\label{thm:3.1} Let $g \in \Bbbk [x] $ $~$be a separable polynomial and $%
\lambda _{0},\ldots,\lambda _{n-1}$ be the roots of $g$ in the algebraic
closure $\overline{\Bbbk }$ of $\Bbbk $ . Then for any polynomials $%
f_{0},~~f_{1},\ldots,~f_{m-1}\in \Bbbk[x]$, there exists unique $r\in \Bbbk[x%
]$ of degree less than $mn$ and unique polynomials $%
q_{0},~q_{1},~...,~q_{m-1}\in \Bbbk[x]$ such that 
\begin{equation*}
f_{i}=r^{( i)}+gq_{i},~0\leq i\leq m-1.
\end{equation*}
This result is optimal, in the sense that if $m\geq 1$ and $g$ is
inseparable, then this result is not true. The polynomial $r$ is given by 
\begin{equation*}
r(x) =\sum_{j=0}^{n-1}X_{j}\Lambda _{j}^{-1}A_{j}\in \Bbbk \left[ x\right],
\end{equation*}
where 
\begin{eqnarray*}
X_{j} &=&\left[ 
\begin{array}{cccccc}
L_{j}(x) & \frac{(x-\lambda _{j}) }{1!}L_{j}( x) & . & . & . & \frac{%
(x-\lambda _{j}) ^{m-1}}{( m-1) !}L_{j}(x)%
\end{array}
\right] , \\
\ A_{j} &=&\left[ 
\begin{array}{cccccc}
f_{0}( \lambda _{j}) & f_{1}( \lambda _{j}) & . & . & . & f_{m-1}( \lambda
_{j})%
\end{array}
\right] ^{T}
\end{eqnarray*}
and $L_{j},\Lambda _{j}$ are respectively as in (\ref{eq:2.1}), (\ref{eq:2.2}%
) by $m_{0}=m_{1}=...=m_{n-1}=m.$
\end{theorem}

\begin{proof}
Let $k$ be the dimension of the field $\Bbbk ( \lambda _{0},\ldots,\lambda
_{n-1}) $ as vector space over $\Bbbk ,$ that is 
\begin{equation}\label{eq:3.1}
k=\left[ \Bbbk ( \lambda _{0},\ldots,\lambda _{n-1}) ,\Bbbk \right]
\end{equation} 
then there exist $\tau _{1},...$,$\tau _{k-1}$ in $\Bbbk ( \lambda
_{0},\ldots,\lambda _{n-1}) $ such that $\left\{ 1,\tau _{1},\ ...,~\tau
_{k-1}\right\} $ is a basis of $\Bbbk ( \lambda _{0},\ldots,\lambda
_{n-1}) $ as a vector space over $\Bbbk .$

Now, using Theorem \ref{thm:2.3} by $m_{0}=m_{1}=...=m_{n-1}=m,$ we have that there
is unique polynomial 
$\widehat{r}\in \Bbbk (\lambda _{0},\ldots,\lambda_{n-1})[x]$ 
given by (\ref{eq:2.5}) for 
$ m_{0}=m_{1}=...=m_{n-1}=m $ having degree less than $mn$ such that 
\begin{equation}\label{eq:3.2}
(\widehat{r})^{( i) }( \lambda _{j}) =f_{i}(
\lambda _{j})
\end{equation} 
for all $0\leq i\leq m-1,0\leq j\leq n-1$. Therefore there exist polynomials 
$\widehat{q}_{i}\in \Bbbk( \lambda _{0},\ldots,\lambda _{n-1})[x]$, $0\leq i\leq m-1$ such that 
\begin{equation}\label{eq:3.3}
\widehat{r}^{(i)}=f_{i}+g\widehat{q}_{i}.
\end{equation} 
Moreover, by (\ref{eq:3.1}) we have that the dimension of 
$\Bbbk(\lambda _{0},\ldots,\lambda _{n-1}) [x]$ as free
module over $\Bbbk [x] $ is $k$ and $\left\{ 1,\tau _{1},\
...,~\tau _{k-1}\right\} $ is a basis of 
$\Bbbk (\lambda_{0},\ldots,\lambda _{n-1}) [x]$ over 
$\Bbbk[x]$. Therefore the polynomials 
$\widehat{r},\widehat{q}_{i}\in \Bbbk(\lambda _{0},\ldots,\lambda _{n-1})[x]$ can be uniquely
written in the following form:
\begin{equation}\label{eq:3.4}
\widehat{r}=r+\sum_{s=1}^{n-1}\tau _{s}r_{s},
\end{equation} 
and 
\begin{equation}\label{eq:3.5}
\widehat{q}_{i}=q_{i}+\sum_{s=1}^{n-1}\tau _{s}q_{si},~0\leq i\leq m-1,
\end{equation} 
where $r, r_{s}, q_{i}, q_{si}\in \Bbbk[x] $ with $\deg r$, $\deg r_{i}<mn.$

Setting (\ref{eq:3.4}) and (\ref{eq:3.5}) in (\ref{eq:3.3})
and using that $\left\{ 1,\tau _{1},\ ...,~\tau _{k-1}\right\} $ is
linearly independent over $\Bbbk[x] $, we get 
\begin{equation}\label{eq:3.6}
r^{( i)}=f_{i}+gq_{i}
\end{equation} 
for all $0\leq i\leq m-1.$

Now from (\ref{eq:3.6}) we clearly get that $r$ satisfies the
identities (\ref{eq:3.2}) and since the polynomial of degree less than 
$mn$ satisfying the identities (\ref{eq:3.2}) is unique, we conclude
that $r=\widehat{r}$.

Finally, suppose that Theorem \ref{thm:3.1} is true for one polynomial $g\in \Bbbk[x]$ 
having a root $\lambda \in \overline{\Bbbk}$ of
multiplicity $>1$. Then there exist a polynomial $g_{0}$ in 
$\overline{\Bbbk}[x]$ such that 
\begin{equation*}
g(x) =(x-\lambda)^{2}g_{0}(x) .
\end{equation*} 
Applying Theorem \ref{thm:3.1} for $g$ and $f_{0}=f_{1}=1$ we obtain 
\begin{equation}\label{eq:3.7}
r(x) =1+(x-\lambda)^{2}g_{0}(x)q_{0}(x),
\end{equation} 
\begin{equation}\label{eq:3.8}
r^{(1)}(x) =1+(x-\lambda)^{2}g_{0}(x) q_{1}(x)
\end{equation} 
for some $r,~q_{0},~q_{1}\in \overline{\Bbbk}[x]$.

Differentiating (\ref{eq:3.7}), and setting $x=\lambda$ in the
resulting identity, as well as, in (\ref{eq:3.8}) we respectively get 
\begin{equation*}
r^{(1)}(\lambda)=0~\text{and}~r^{(1)}(\lambda) =1\,.
\end{equation*} 
That is not true.
\end{proof}

\begin{corollary}
\label{cor:3.2} Let $g$ be as in Theorem \ref{thm:3.1}. Then for any $f\in
\Bbbk [x]$, $c\in \Bbbk -\left\{0\right\}$, $m\in \mathbb{N}$ with $m\geq 2$
there exists unique $r\in \Bbbk[x]$ of degree less than $mn$ such that $r=f%
\mod g$ and $r^{\prime }=crg^{m-1}$.
\end{corollary}

\begin{proof}
According to Theorem \ref{thm:3.1} there exists unique $r\in \Bbbk[x]$
of degree less than $mn$, such that 
\begin{equation}\label{eq:3.9}
f=r, cf=r^{\prime}\mod g,\ldots,c^{m-1}f=r^{(m-1)}\mod g.
\end{equation} 
Now multiplying the $i-$th relation of (\ref{eq:3.9}) by $c$ and
afterwards substracting from the result the $i+1-$th relation we get 
\begin{equation*}
(r^{\prime }-cr)^{(i)}=0\mod g,
\end{equation*} 
for all $0\leq i\leq m-2$ and since $g$ is separable we direct get the
conclusion.
\end{proof}

\section{An explicit formula for $S_{A}$}

We will use Theorem \ref{thm:3.1} to give a closed formula for the
semi--simple part of the Jordan decomposition of an algebraic number of an
algebra \textbf{A} over a perfect field $\Bbbk$.

\begin{lemma}
\label{lem:4.1} Let $g\in \Bbbk[x]$ be a separable polynomial and $%
\lambda_{0},\ldots,\lambda _{n-1}$ are the roots of $g$ in $\overline{\Bbbk }
$. Let $f\in \Bbbk[x,y]$ be a polynomial of two variables and $m$ be a
positive integer. Then there exists unique $r\in \Bbbk[x]$ of degree less
than $mn$ such that 
\begin{equation*}
f(x,y) =r(x+y)\mod I,
\end{equation*}
where $I\subset \Bbbk[ x,y]$ is the ideal generated from the polynomials $%
g(x) $, $y^{m}$. Further the polynomial $r$ is given by 
\begin{equation}  \label{eq:4.1}
r(x) =\sum_{j=0}^{n-1}X_{j}\Lambda _{j}^{-1}A_{j}\in \Bbbk[x],
\end{equation}
where the matrices $X_{j}$, $A_{j}$ are given by 
\begin{eqnarray*}
X_{j} &=&\left[ 
\begin{array}{cccccc}
L_{j}(x) & \frac{(x-\lambda _{j}) }{1!}L_{j}( x) & . & . & . & \frac{%
(x-\lambda _{j}) ^{m-1}}{( m-1) !}L_{j}(x)%
\end{array}
\right] , \\
A_{j} &=&\left[ 
\begin{array}{cccccc}
f( \lambda _{j},0) & \frac{\partial }{\partial y}f( \lambda _{j},0) & . & .
& . & \frac{\partial ^{m-1}}{\partial y^{m-1}}f( \lambda _{j},0)%
\end{array}
\right] ^{T},
\end{eqnarray*}
and $L_{j},\Lambda _{j}$ are as in (\ref{eq:2.1}), (\ref{eq:2.2}) by $%
m_{0}=m_{1}=...=m_{n-1}=m.$
\end{lemma}

\begin{proof}
The polynomial $f$ can be rewritten in the following form 
\begin{equation}\label{eq:4.2}
f(x,y) =\sum_{i=0}^{m-1}\frac{1}{i!}\frac{\partial^{i}}{ 
\partial y^{i}}f(x,0) y^{i}+y^{m}h(x,y) ,
\end{equation} 
for some $h\in \Bbbk[x,y]$.

Furthermore, according to Theorem \ref{thm:3.1} there exists unique polynomial 
$r\in\Bbbk[x]$ of degree less than $mn$, given by (\ref{eq:4.1}), such that: 
\begin{equation}\label{eq:4.3}
\frac{\partial^{i}}{\partial y^{i}}f(x,0) =r^{( i)}(x)\mod g(x), ~0\leq i\leq m-1,
\end{equation}

Combining (\ref{eq:4.2}) with (\ref{eq:4.3}) we get 
\begin{equation}\label{eq:4.4}
f(x,y) =\sum_{i=0}^{m-1}\frac{r^{( i)}(x)}{i!}y^{i}\mod I.
\end{equation}
Furthermore, by using the Taylor formula we have: 
\begin{equation}\label{eq:4.5} 
\sum_{i=0}^{m-1}\frac{r^{( i)}(x) }{i!}y^{i}=r(x+y)\mod y^{m}.
\end{equation} 
Substituting (\ref{eq:4.5}) in (\ref{eq:4.4}), completes the
required proof.
\end{proof}

Let $p\in $ $\Bbbk[x]$ be the minimal polynomial of an algebraic element $A$
of an algebra $\mathbf{A}$ over $\Bbbk$ with unit $1$. Let $\lambda _{i}$, $%
i=0,1,\ldots,n-1$ be the distinct roots of $p$ in $\overline{\Bbbk}$, and
let $k_{i}$, $i=1,2,\ldots,n-1$ be their respective multiplicities. We
denote $\widehat{p}(x):=\prod_{i=0}^{n-1}(x-\lambda _{i})$ and $m(p):=\max
\left\{ k_{0},\ldots,k_{n-1}\right\}$.

\begin{theorem}
\label{thm:4.2} The semi--simple part $S_{A}$ of the Jordan decomposition of 
$A$ is given by $S_{A}=r(A) $, where $r(x)$ is the polynomial 
\begin{equation}  \label{eq:4.6}
r(x)=\sum_{j=0}^{n-1}X_{j}\Lambda _{j}^{-1}A_{j}\in \Bbbk[x],
\end{equation}
where 
\begin{eqnarray*}
X_{j} &=&\left[ 
\begin{array}{cccccc}
L_{j}(x) & \frac{(x-\lambda _{j}) }{1!}L_{j}( x) & . & . & . & \frac{%
(x-\lambda _{j}) ^{m-1}}{( m-1) !}L_{j}(x)%
\end{array}
\right] , \\
\ A_{j} &=&\left[ 
\begin{array}{cccccc}
\lambda _{j} & 0 & . & . & . & 0%
\end{array}
\right] ^{T},
\end{eqnarray*}
and $L_{j},\Lambda _{j}$ are as in (\ref{eq:2.1}), (\ref{eq:2.2}) by $%
m_{0}=m_{1}=...=m_{n-1}=m(p)$.
\end{theorem}

\begin{proof}
Since $S_{A}$ is the semi--simple part of $A$ and $N_{A}$ is the nilpotent
part of $A$, the minimal polynomials of $S_{A}$ and $N_{A}$ are respectively 
$\widehat{p}(x)$ and $x^{m(p)}$. So we have $\widehat{p}(S_{A})=0$ 
and $N_{A}^{m( p)}=0$. Now if we apply Lemma \ref{lem:4.1}  by choosing 
$f(x,y)=x$, $g(x)=\widehat{p}(x)$ and $m=m(p)$, and taking into account that 
$f(x,0)=x$, and $\frac{\partial^{i}}{\partial y^{i}}f(x,0)=0$ 
for $~1\leq i\leq m-1$ we have that for the polynomial $r$ given by 
(\ref{eq:4.6}) holds: 
\begin{equation}\label{eq:4.7}
x=r(x+y)\mod I~,
\end{equation} 
where $I$ is the ideal generated from the polynomials $\widehat{p}(x)$ and 
$y^{m(p)}$. Finally setting $x=S_{A}$ and $y=N_{A}$ in (\ref{eq:4.7}) we get 
$S_{A}=r(S_{A}+N_{A})=r(A)$.
\end{proof}

\subsection{ An implementation of the explicit formula for $S_{A}$ in the $%
\mathtt{MATLAB}$ programming environment and some arithmetic examples}

In this subsection we present a computational implementation in the MATLAB
programming environment for the calculation of the semi-simple part $S_{A}$
of the jordan decomposition of $A$. The MATLAB source code of function 
\texttt{semisimple} is provided, with appropriate exlanatory comments and
pointers to the relevant Equations. Instead of using the interpolating
polynomial given in Theorem \ref{thm:4.2}, we implemented the equivalent
formulation for the interpolating polynomial of Eq. (\ref{eq:2.6}), by
setting 
\begin{equation*}
[c_{0j},c_{1j},\ldots,c_{m( p)-1j}]^{T}=\Lambda_{j}^{-1}[\lambda
_{j},0,\ldots,0]^{T}.
\end{equation*}
The function's input arguments are the matrix $A$ whose semisimple part is
required, as well as the number of decimal digits that will be taken into
consideration for calculating the multiplicity of the eigenvalues of $A$.
Integer $m(p)$ in Theorem \ref{thm:4.2} which is equal to the maximum
multiplicity of the roots of the minimal polynomial of $A$, is calculated as
the maximum number of non zero elements of the rows and columns of the
Jordan decomposition of $A$. The main function (\texttt{semisimple}) uses
two more functions that are also provided, one for the calculation of the
Lagrange polynomial coefficients and one for the calculation of the
coefficients $n^{th}$ power of a given polynomial.
\begin{verbatim}
 
function R=semisimple(A,tol)
% by A. Kechriniotis, K. Delibasis, 2009
% A: matrix whose semisimple part is required
% tol: number of digits to use for finding multiple eigenvalues of A
 
p=poly(A);
a=roots(p);
 
% determine multiplicity of roots according to tolerance tol
tol=10^tol;
r1=round(a*tol)/tol;
[r_uniq,i_uniq]=unique(r1);
for i=1:length(r_uniq)
    multi_uniq(i)=sum(r1==r_uniq(i));
end
a=r_uniq;
n=length(a);
 
% determine m as the max root multiplicity of the minimal polynomial
J=jordan(A);
m=max(max(sum(J~=0)), max(sum(J~=0,2)));
 
% create a table with n rows to hold Langrange polynomial coefficients 
as in (2.1)
for k=1:n
    myLangr_coefs(k,:)=myLangr(a,k,n);
end
 
% Calculate cnm array in (2.6)
for j0=1:n
      generate Lamda mxm matrix, as in (2.2)
    for i=1:m
        for j=1:m
            if j>i
                Lamda(i,j)=0;
            else
                l1=myLangr_coefs(j0,:);
                l1=polypower(l1,m);
                f1=nchoosek(i-1,j-1);
                if i>1     calc (i-1) order derivative of l1
                
                    lk=l1;
                    for k=1:i-j   
                        lk=polyder(lk);
                    end
                    Lamda(i,j)=f1*polyval(lk,a(j0));
                else     no derivative of l1
                    Lamda(i,j)=f1*polyval(l1,a(j0));
                end
            end
        end
    end
    I=diag(ones(size(A,1),1));
    col1=[a(j0),zeros(1,m-1)]';
    c(:,j0)=(Lamda^-1)*col1;
    col1;
end
 
R=zeros(size(A));
for i=1:m
    for j=1:n
        Ln=myLangr_coefs(j,:);
        LnM=polypower(Ln,m);
        R=R+c(i,j)*polyvalm(LnM,A)*((A-a(j)*I)^(i-1))/factorial(i-1);
    end
end
 
 
function Ln=myLangr(a,n,N)
% a: is the vector of roots
% N: length of a
% n: integer between 1 and N inclusive
% Ln is an array of coefficients of the Lagrange polynomial
Ldenom=1;
Lnom=1;
for i=1:N
    if i~=n
       Ldenom=Ldenom*(a(n)-a(i));
       Lnom=conv(Lnom,[1,-a(i)]);
    end
end
Ln=Lnom/Ldenom;
 
 
function p1=polypower(p,n);
%p: an array of coefficients of the Lagrange polynomial
% n: integer power
% p1: array of coefficients of polynomial p raised to the nth power
p1=p;
if n>1
    for i=1:n-1
        p1=conv(p1,p);
    end
end
 
\end{verbatim}

In order to verify the accuracy of the above implementation, we present the
following arithmetic examples of the application of the above source code
for the calculation of the semisimple part $S_{A}$, taken from \cite%
{Hsieh-Kohno-Sibuya}.

Let $A=\left[ 
\begin{array}{ccc}
3 & 4 & 3 \\ 
2 & 7 & 4 \\ 
-4 & 8 & 3%
\end{array}
\right] $. It is easy to verify that $A$ has 2 unique eigenvalues: $%
a_{0}=1,a_{1}=11$ with multiplicity 2 and 1 respectively. The matrices $%
\Lambda _{0}, \Lambda _{1}$ according to Eq. (\ref{eq:2.2}) are calculated
as $\Lambda _{0}=\left[ 
\begin{array}{cc}
1 & 0 \\ 
-0.2 & 1%
\end{array}
\right] $ and $\Lambda _{1}=\left[ 
\begin{array}{cc}
1 & 0 \\ 
0.2 & 1%
\end{array}
\right]$, whereas the matrix $c_{ij}$ obtains the following values
according, to Eq. (\ref{eq:2.6}): $c=\left[ 
\begin{array}{cc}
1 & 11 \\ 
0.2 & -2.2%
\end{array}
\right]$. The semisimple part is calculated according to Eq. (\ref{eq:2.6}): 
$S_{A}=\left[ 
\begin{array}{ccc}
1 & 5.6 & 2.8 \\ 
0 & 8.6 & 3.8 \\ 
0 & 4.8 & 3.4%
\end{array}
\right] $. It is easy to demonstrate that $S_{A}$ is semisimple (e.g. by
calculating the Jordan decomposition of $S_{A}$ which is a diagonal $3\times
3$ matrix), as well as that $( A-S_{A})$ is nilpotent, since $(A-S_{A})^{2}$
is the $3\times 3$ zero matrix. The result obtained for $S_{A}$ by the
proposed method is identical to the result produced in \cite%
{Hsieh-Kohno-Sibuya}.

The second arithmetic example is also taken from \cite{Hsieh-Kohno-Sibuya}.
Let 
\begin{equation*}
A=\left[ 
\begin{array}{cccccccccc}
0 & -1 & 0 & 0 & 0 & 0 & 0 & 0 & 0 & 0 \\ 
-1 & 0 & 0 & 0 & 0 & 0 & 0 & 0 & 0 & 0 \\ 
0 & 0 & 1 & -1 & 0 & 0 & 0 & 0 & 0 & 0 \\ 
0.25 & 0 & -1 & 1 & 0 & 0 & 0 & 0 & 0 & 0 \\ 
0 & 0 & 0 & 0 & 2 & -1 & 0 & 0 & 0 & 0 \\ 
0 & 0 & 0.25 & 0 & -1 & 2 & 0 & 0 & 0 & 0 \\ 
0 & 0 & 0 & 0 & 0 & 0 & 3 & -1 & 0 & 0 \\ 
0 & 0 & 0 & 0 & 0.25 & 0 & -1 & 3 & 0 & 0 \\ 
0 & 0 & 0 & 0 & 0 & 0 & 0 & 0 & 4 & -1 \\ 
0 & 0 & 0 & 0 & 0 & 0 & 0.25 & 0 & -1 & 4%
\end{array}
\right] .
\end{equation*}
It is easy to show that the roots of the characteristic polynomial of $A$
are the following: $a_{0}=5,a_{1}=4,a_{2}=-1,a_{3}=3$ (multiplicity equal to
2)$,a_{4}=2$ (multiplicity equal to 2), $a_{5}=1$ (multiplicity equal to 2)$%
,a_{6}=0$. The matrices $\Lambda _{i}~i=0,\ldots,6$ according to Eq. (\ref%
{eq:2.2}) are calculated as:

$\Lambda _{0}=\left[ 
\begin{array}{cc}
1 & 0 \\ 
-4.9 & 1%
\end{array}
\right] $, $\Lambda _{1}=\left[ 
\begin{array}{cc}
1 & 0 \\ 
-2.5667 & 1%
\end{array}
\right] $, $\Lambda _{2}=\left[ 
\begin{array}{cc}
1 & 0 \\ 
-1.667 & 1%
\end{array}
\right] $,

$\Lambda _{3}=\left[ 
\begin{array}{cc}
1 & 0 \\ 
0 & 1%
\end{array}
\right] $, $\Lambda _{4}=\left[ 
\begin{array}{cc}
1 & 0 \\ 
1.1667 & 1%
\end{array}
\right] $, $\Lambda _{5}=\left[ 
\begin{array}{cc}
1 & 0 \\ 
2.5667 & 1%
\end{array}
\right] $, $\Lambda _{6}=\left[ 
\begin{array}{cc}
1 & 0 \\ 
4.9 & 1%
\end{array}
\right]$. \noindent The matrix $c_{ij}$ obtains the following values
according to Eq. (\ref{eq:2.6}):

$c=\left[ 
\begin{array}{ccccccc}
-1 & 0 & 1 & 2 & 3 & 4 & 5 \\ 
-4.9 & 0 & 1.1667 & 0 & -3.5 & -10.2667 & -24.5%
\end{array}
\right]$. \noindent The semisimple part $S_{A}$ is calculated according to
Eq. (\ref{eq:2.6}): 
\begin{equation*}
\left[ 
\begin{array}{cccccccccc}
0 & -1 & 0 & 0 & 0 & 0 & 0 & 0 & 0 & 0 \\ 
-1 & 0 & 0 & 0 & 0 & 0 & 0 & 0 & 0 & 0 \\ 
0 & 0 & 1 & -1 & 0 & 0 & 0 & 0 & 0 & 0 \\ 
0.25 & 0 & -1 & 1 & 0 & 0 & 0 & 0 & 0 & 0 \\ 
-0.0156 & 0.0156 & 0 & 0 & 2 & -1 & 0 & 0 & 0 & 0 \\ 
-0.0156 & 0.0156 & 0.25 & 0 & -1 & 2 & 0 & 0 & 0 & 0 \\ 
0.0039 & -0.0026 & -0.0156 & 0.0156 & 0 & 0 & 3 & -1 & 0 & 0 \\ 
0.0052 & -0.0039 & -0.0156 & 0.0156 & 0.25 & 0 & -1 & 3 & 0 & 0 \\ 
-0.0004 & 0.0002 & 0.0039 & -0.0026 & 0.0156 & 0.0156 & 0 & 0 & 4 & -1 \\ 
-0.0007 & 0.0004 & 0.0052 & 0.0039 & 0.0156 & 0.0156 & 0.25 & 0 & -1 & 4%
\end{array}
\right]
\end{equation*}
Again this result is identical with Example $3$ from $[5]$. It is also easy
to verify that the nilpotency holds, since $( A-S_{A}) ^{2}=0.$

\section{Further generalizations of Theorem 3.1 and derived identities}

The main result of this section is the generalization of Theorem \ref%
{thm:3.1}, which is expressed in the following Theorem.

\begin{theorem}
\label{thm:5.1} Let $g \in \Bbbk[x]$ be separable of degree $n$. Let $\Pi
~\in (\Bbbk[x])^{m \times m}$ such that $\det (\Pi) \neq 0$. Then for any
polynomials $f_{0}$, $f_{1\text{ }}$,\ldots,$~f_{m-1} \in \Bbbk[x]$ there
exist unique polynomial $r\in\Bbbk[x]$ of degree less than $mn$, such that 
\begin{equation*}
\Pi \left[ 
\begin{array}{c}
r \\ 
r^{\prime } \\ 
. \\ 
. \\ 
. \\ 
r^{( m-1) }%
\end{array}
\right] =E\left[ 
\begin{array}{c}
f_{0} \\ 
f_{1} \\ 
. \\ 
. \\ 
. \\ 
f_{m-1}%
\end{array}
\right] \mod g,
\end{equation*}
where $E:=\gcd (g,\ \det (\Pi))$.
\end{theorem}

\begin{proof}
We start from 
\begin{equation}\label{eq:5.1}
\Pi \widetilde{\Pi }=\widetilde{\Pi }\Pi =\det ( \Pi ) I_{M},
\end{equation} 
where $\widetilde{\Pi }$ is the adjugate of $\Pi$ and $I_{m}$ is the $ 
m\times m$ unit matrix.

Moreover, since $E=\gcd ( g,\det (\Pi))$ one has
that there exist $H,G$ $\in \Bbbk[x] $ such that 
\begin{equation}\label{eq:5.2}
H \det(\Pi)+Gg=E.
\end{equation} 
Combining (\ref{eq:5.1}) with (\ref{eq:5.2}) we get, 
\begin{equation}\label{eq:5.3}
H \Pi \widetilde{\Pi}=( E-gG) I_{m}.
\end{equation} 
Now, according to Theorem \ref{thm:3.1} there exist unique $r\in \Bbbk[x]$ of 
degree less than $mn$ such that 
\begin{equation}\label{eq:5.4}
\left[ 
\begin{array}{c}
r \\ 
r^{\prime } \\ 
. \\ 
. \\ 
. \\ 
r^{( m-1) } 
\end{array} 
\right] =H\widetilde{\Pi }\left[ 
\begin{array}{c}
f_{0} \\ 
f_{1} \\ 
. \\ 
. \\ 
. \\ 
f_{m-1} 
\end{array} 
\right] \mod g.
\end{equation} 
Multiplying (\ref{eq:5.4}) from the left by $\Pi$ and setting (\ref{eq:5.3})
in the resulting identity we get the conclusion.
\end{proof}

\begin{remark}
\label{rem:5.2} Choosing $\Pi=I_{m}$ in Theorem \ref{thm:5.1} we get Theorem %
\ref{thm:3.1}. Therefore Theorem \ref{thm:3.1} can be regarded as a
generalization of Theorem \ref{thm:5.1}.
\end{remark}

Now we will apply Theorem \ref{thm:5.1} to produce some formulas for
polynomials involving derivatives.

\begin{corollary}
\label{cor:5.3} Let $g,$ $g_{0},~g_{1},\ldots,g_{m-1}\in \Bbbk[x]$ be
polynomials. Assume that $g$ is separable and that 
\begin{equation*}
( g_{i},g) =1,~0\leq i\leq m-1.
\end{equation*}
Then for any $f_{0},~f_{1},\ldots,f_{m-1}\in \Bbbk[x]$ there exists unique $%
r\in \Bbbk[x]$ of degree less than $mn,$ such that 
\begin{equation*}
g_{i}r^{( i) }=f_{i}\mod g,~0\leq i\leq m-1.
\end{equation*}
\end{corollary}

\begin{proof}
Applying Theorem \ref{thm:5.1} by 
\begin{equation*}
\Pi :=\left[ 
\begin{array}{cccccc}
g_{0} & 0 & . & . & . & 0 \\ 
0 & g_{1} & . & . & . & 0 \\ 
. &  & . &  &  & . \\ 
. &  &  & . &  & . \\ 
. &  &  &  & . & . \\ 
0 & . & . & . & 0 & g_{m-1} 
\end{array} 
\right] ,
\end{equation*} 
and afterwards using that 
$E=\gcd (g,\det(\Pi)=\prod_{i=0}^{m-1}g_{i}) =1$ 
we directly get the conclusion.
\end{proof}

\begin{corollary}
\label{cor:5.4} Let $g$, $g_{0}, g_{1}, \ldots, g_{m-1} \in \Bbbk[x]$ be as
in Corollary \ref{cor:5.3}. Then for any $f_{0}, f_{1}, \ldots,f_{m-1}\in 
\Bbbk[x]$ there exists unique $r$ of degree less than $mn$ such that 
\begin{equation*}
(g_{i}r)^{(i)}=f_{i}\mod g,~0 \leq i\leq m-1.
\end{equation*}
\end{corollary}

\begin{proof}
Let $\Pi \in (\Bbbk[x])^{m \times m}$ 
be the matrix defined by 
\begin{equation}\label{eq:5.5}
\Pi =\left[ 
\begin{array}{cccccc}
g_{0} & 0 & . & . & . & 0 \\ 
\binom{1}{0}g_{1}^{( 1) } & \binom{1}{1}g_{1} & . & . & . & 0 \\ 
. & . & . &  &  & . \\ 
. & . &  & . &  & . \\ 
. & . &  &  & . & . \\ 
\binom{m-1}{0}g_{m-1}^{( m-1) } & \binom{m-1}{1}g_{m-1}^{(
m-2) } & . & . & . & \binom{m-1}{m-1}g_{m-1}^{( 0) } 
\end{array} 
\right] .
\end{equation} 
From the assumptions 
\begin{equation*}
(g_{i},g) =1,~0\leq i\leq m-1
\end{equation*}
we have 
\begin{equation}\label{eq:5.6}
(\det(\Pi),g)=1.
\end{equation} 
Applying Theorem \ref{eq:5.1} for $\Pi$, as given in (\ref{eq:5.5}), using  
(\ref{eq:5.6}) and the Leibnitz's rule, we get the conclusion.
\end{proof}

\end{document}